\documentclass[reqno,11pt]{amsart}        
\usepackage{amsfonts,amsmath,amssymb,amsthm,colordvi,mathrsfs,graphicx}
\usepackage{caption,subcaption,float}
\usepackage[utf8]{inputenc}
\usepackage{mathrsfs}
\usepackage{geometry}
\usepackage{enumerate}

\usepackage{amsmath, amssymb, amsthm, geometry, hyperref}

\usepackage{tabularx}
\usepackage{tabulary}
 
\hypersetup{
  colorlinks=true,
  linkcolor=blue,
  citecolor=blue,
  urlcolor=blue
}
\usepackage[all,knot,poly]{xy}
\usepackage{tikz-cd}
\usepackage{tikz}
\usepackage{verbatim}
\usetikzlibrary{arrows}
\usetikzlibrary{knots}
\tikzset{every path/.style={line width=0.4pt},every node/.style={transform shape,knot crossing,inner sep=1.5pt},>=triangle 60,text node/.style={rectangle,transform shape=false,black}}
 \usetikzlibrary{decorations.markings, arrows.meta}

\theoremstyle{plain}      
\newtheorem{thm}{Theorem}[section]     
\newtheorem{theorem}[thm]{\bf Theorem}

\newtheorem{proposition}[thm]{\bf Proposition}

\theoremstyle{remark}

\theoremstyle{definition}      
     
\begin{document}

%\title[Semiregularity and Maximal Nodal Deformations of Curves]{Semiregularity and Maximal Nodal Deformations of Curves in Linear Systems}

\author{Mounir Nisse}
 
\address{Mounir Nisse\\
Department of Mathematics, Xiamen University Malaysia, Jalan Sunsuria, Bandar Sunsuria, 43900, Sepang, Selangor, Malaysia.
}
\email{mounir.nisse@gmail.com, mounir.nisse@xmu.edu.my}
\thanks{}
\thanks{This research is supported in part by Xiamen University Malaysia Research Fund (Grant no. XMUMRF/ 2020-C5/IMAT/0013).}
%\maketitle

%\tableofcontents

%\title{Ideas Toward Deformations of Curves to Nodal Curves}

%\title{Maximal Nodal Deformations of Curves on Degenerating Surfaces}%{Deforming Curves to Maximal Nodal Configurations} %Equigeneric Deformations and Nodal Curves on Surfaces
\title{Equisingular Deformations of Curves and Surfaces in Threefolds}

\maketitle
 
\begin{abstract}
We study equisingular deformation problems for curves and surfaces in algebraic
families, with particular emphasis on situations where nodal behavior is no longer
generic. Extending classical Severi theory, we develop deformation--theoretic criteria
ensuring the existence of deformations with isolated singularities of minimal type,
including cusps on curves and ordinary double points on curves and surfaces in
threefolds. Under unobstructedness and surjectivity assumptions for natural
global--to--local maps of normal bundles, we prove maximality results showing that the
number of such singularities is governed by the global realizability of equisingular
deformation directions rather than by numerical invariants alone. Logarithmic
semiregularity allows these results to persist in degenerations with normal crossings
special fibers. We further explain how these singularities arise as boundary phenomena
of equigeneric Severi strata and outline applications to refined Severi counts via
logarithmic and tropical methods.
\end{abstract}

\section{Introduction}

The deformation theory of singular subvarieties in algebraic families is a central
topic in algebraic geometry, with deep connections to Severi theory, enumerative
geometry, and degeneration techniques. Classical Severi theory studies nodal curves on
smooth surfaces and explains why, under equigeneric deformations, nodes appear as the
generic singularities and their number is governed purely by the arithmetic genus.
More recently, this philosophy has been extended to degenerations of surfaces via
logarithmic geometry, leading to conceptual proofs of generic nodality and maximality
results in highly singular settings.

The present work develops a systematic generalization of this picture in two new
directions. First, it replaces nodal curves on surfaces by more rigid singularities,
such as cuspidal curves and nodal curves lying on threefolds. Second, it replaces
equigeneric deformation theory by equisingular deformation theory as the correct
framework whenever nodal behavior is no longer generic. These two shifts reveal a
common underlying principle: once singularities cease to be generic, maximality is no
longer controlled by numerical invariants alone, but by the global realizability of
local equisingular deformation directions.

A central contribution of this work is the formulation and proof of equisingular
deformation theorems for curves and surfaces inside higher--dimensional ambient spaces.
For curves on surfaces, cusps replace nodes as the simplest non--generic singularities.
We show that cuspidal curves occur as boundary points of equigeneric Severi strata and
prove precise criteria ensuring the existence of deformations preserving cusps.
Maximality of cuspidal behavior is shown to be governed by the surjectivity of the
global--to--local map from sections of the normal bundle to equisingular deformation
spaces, rather than by the $\delta$--invariant alone.

The main new results concern curves and surfaces inside threefolds. For curves in
threefolds, nodal singularities are no longer generic: a general deformation smooths
all singularities, and nodes appear only under equisingular constraints. We establish
a nodal deformation theorem for curves in threefolds showing that, under unobstructed
embedded deformation theory and surjectivity of the normal bundle restriction map, one
can realize the maximal number of nodes compatible with global deformation theory.
This places nodal curves in threefolds in the same conceptual class as cuspidal curves
on surfaces.

The most substantial new contribution is the extension of this philosophy to surfaces
inside threefolds. We prove a nodal surface deformation theorem asserting that, under
unobstructedness and surjectivity assumptions, a surface in the special fiber of a
threefold degeneration admits deformations whose general members have only ordinary
double points, with the maximal possible number of such singularities. This result
provides a higher--dimensional analogue of classical Severi theory and clarifies the
role of ordinary double points as the simplest isolated surface singularities arising
in deformation problems.

Throughout the paper, logarithmic geometry plays a crucial role in treating
degenerations. By replacing the normal bundle with the logarithmic normal bundle and
unobstructedness with logarithmic semiregularity, we obtain deformation theorems that
remain valid in families with normal crossings special fibers. This allows us to
interpret nodal and cuspidal loci as genuine boundary strata of equigeneric deformation
spaces and ensures that maximality statements persist under degeneration.

Finally, we connect these deformation--theoretic results with enumerative geometry.
Cuspidal and nodal conditions are interpreted as refined Severi conditions, modifying
classical Severi degrees. Logarithmic and tropical techniques provide effective tools
for computing these refined invariants, with cusps and higher--dimensional nodes
appearing as limits of colliding simpler singularities.

Taken together, the results of this work provide a unified deformation--theoretic
framework for understanding nodal and cuspidal phenomena across dimensions. They
clarify when singularities are generic, when they are rigid, and how maximal singular
behavior is governed by global geometry rather than purely numerical constraints.

%%%%%%%%%%%%%%%%%%%%%%%%%%%%%%%%%%%%%%%%%%%%%%%%%%%%%%%%%%%%%%%%%%%%

\section{Preliminaries}

We collect in this section the basic notions and results on deformation theory,
singularities, and logarithmic geometry that will be used throughout this work. The
purpose is not to give exhaustive proofs, but to fix notation and recall the precise
form of the deformation--theoretic tools underlying our results.

Let $k$ be an algebraically closed field of characteristic zero. All schemes and
varieties are assumed to be separated and of finite type over $k$. A family over a
disk will mean a flat morphism $\pi:\mathcal X\to\Delta$, where $\Delta$ is the
spectrum of a discrete valuation ring or a small analytic disk.

\subsection*{Local Complete Intersections and Normal Bundles}

Let $Y$ be a smooth variety and let $Z\subset Y$ be a closed subscheme. The subscheme
$Z$ is called a local complete intersection if, locally on $Y$, its ideal sheaf is
generated by a regular sequence. In this case the conormal sheaf
\[
\mathcal I_Z/\mathcal I_Z^2
\]
is locally free, and its dual
\[
N_{Z/Y}=\mathcal Hom(\mathcal I_Z/\mathcal I_Z^2,\mathcal O_Z)
\]
is called the normal bundle of $Z$ in $Y$.

Infinitesimal embedded deformations of $Z$ inside $Y$ are parametrized by
$H^0(Z,N_{Z/Y})$, while obstructions lie in $H^1(Z,N_{Z/Y})$. When
$H^1(Z,N_{Z/Y})=0$, all infinitesimal deformations integrate to actual deformations,
and $Z$ is said to be unobstructed as an embedded subvariety of $Y$.

\subsection*{Singularities and Local Deformation Spaces}

Let $(Z,p)$ be an isolated singularity of a reduced local complete intersection
subscheme. Its local deformation space is governed by the finite--dimensional vector
space
\[
T^1_{Z,p}=\mathrm{Ext}^1(L_{Z,p},\mathcal O_{Z,p}),
\]
where $L_{Z,p}$ denotes the cotangent complex. Elements of $T^1_{Z,p}$ parametrize
first--order deformations of the germ $(Z,p)$.

For planar curve singularities, the $\delta$--invariant measures the discrepancy
between the arithmetic and geometric genera and is upper semicontinuous in flat
families. Ordinary nodes and cusps are characterized by their $\delta$--invariants and
by the dimensions of their local deformation spaces. An ordinary node has
$\delta=1$ and $\dim T^1=1$, while an ordinary cusp has $\delta=1$ and $\dim T^1=2$.
The equisingular deformation space $T^{1,\mathrm{es}}_{Z,p}$ is the linear subspace of
$T^1_{Z,p}$ consisting of deformations preserving the analytic type of the singularity.

For surfaces, the simplest isolated singularity is the ordinary double point, locally
given by $x^2+y^2+z^2=0$. Its local deformation space is one--dimensional and has no
moduli, so every nonzero deformation direction preserves the analytic type.

\subsection*{Equigeneric and Equisingular Deformations}

A deformation of a curve is called equigeneric if it preserves the arithmetic genus,
or equivalently the total $\delta$--invariant. Equigeneric deformations form a closed
subset of the base of the semiuniversal deformation and are generically nodal for
curves on surfaces. In contrast, equisingular deformations preserve the analytic type
of each singularity and typically form a proper closed subset of the equigeneric
locus.

For curves on surfaces, nodes are generic in equigeneric families, while cusps occur
as boundary points corresponding to collisions of nodes. For curves in threefolds and
surfaces in threefolds, nodal singularities are no longer generic, and equisingular
deformation theory provides the correct framework.

\subsection*{Logarithmic Geometry}

Let $\pi:\mathcal X\to\Delta$ be a flat family with normal crossings special fiber
$\mathcal X_0$. Endowing $\mathcal X$ with the divisorial logarithmic structure induced
by $\mathcal X_0$ and $\Delta$ with the standard logarithmic point, the morphism $\pi$
becomes logarithmically smooth.

If $Z\subset\mathcal X_0$ is a local complete intersection meeting the singular locus
transversely, its logarithmic embedded deformation theory is governed by the
logarithmic normal bundle $N^{\log}_{Z/\mathcal X_0}$. Infinitesimal logarithmic
deformations are parametrized by $H^0(Z,N^{\log}_{Z/\mathcal X_0})$, with obstructions
in $H^1(Z,N^{\log}_{Z/\mathcal X_0})$.

The logarithmic normal bundle fits into a canonical exact sequence
\[
0\longrightarrow N^{\log}_{Z/\mathcal X_0}
\longrightarrow N^{\log}_{Z/\mathcal X}
\longrightarrow \mathcal O_Z
\longrightarrow 0,
\]
whose extension class is the restriction of the logarithmic Kodaira--Spencer class.
Logarithmic semiregularity means that the induced map
\[
H^1(Z,N^{\log}_{Z/\mathcal X_0})\longrightarrow
H^2(\mathcal X_0,\mathcal O_{\mathcal X_0})
\]
is injective. Under this assumption, logarithmic obstructions vanish and all
infinitesimal logarithmic deformations lift to the total space.

\subsection*{Severi Strata and Boundary Phenomena}

Given a linear system $|L|$ on a smooth surface, Severi varieties parametrize curves
with prescribed singularities. Equigeneric Severi strata are dominated by nodal curves,
while loci of curves with higher singularities such as cusps form boundary strata of
higher codimension. This philosophy extends to higher dimensions, where nodal curves
and nodal surfaces appear as rigid boundary points in deformation spaces.

These preliminaries provide the deformation--theoretic framework for the results
proved in the subsequent sections.
 
 %%%%%%%%%%%%%%%%%%%%%%%%%%%%%%%%%%%%%%%%%%%%%%%%%%%%%%%%%%%%%%%%%%%%%
%%%%%%%%%%%%%%%%%%%%%%%%%%%%%%%%%%%%%%%%%%%%%%%%%%%%%%%%%%%%%%%%%%%%%

 \section{Cuspidal deformations in a linear system}

Replacing nodes by cusps fundamentally changes both the deformation problem and the
meaning of ``maximality'', because cusps are not the simplest singularities compatible
with genus loss. The reformulation therefore requires additional structure.

Let $\pi:\mathcal X\to\Delta$ be a flat family over a disk, with special fiber
$\mathcal X_0$ and smooth general fiber $\mathcal X_t$ for $t\neq0$, and let
$C\subset\mathcal X_0$ be a reduced curve. We now ask under which conditions $C$ can
be deformed to a curve $C_t\subset\mathcal X_t$ whose singularities are ordinary
cusps, and how many cusps such a deformation can have.

An ordinary cusp is a planar singularity analytically equivalent to $y^2=x^3$. Its
$\delta$--invariant equals $1$, just as for a node, but its local deformation space
$T^1$ has dimension $2$ rather than $1$. Consequently, cusps impose strictly more
conditions on deformations than nodes, even though they contribute the same amount
to the genus drop. This difference has two essential consequences.

First, the numerical invariant controlling the problem is still the total
$\delta$--invariant
\[
\delta(C)=\sum_{p\in\mathrm{Sing}(C)}\delta_p,
\]
since the arithmetic genus satisfies
\[
p_a(C)=g(\widetilde C)+\delta(C)
\]
and is invariant in flat families. Any deformation $C_t$ whose singularities are all
cusps must therefore satisfy
\[
\#\{\text{cusps of }C_t\}\le \delta(C).
\]
However, unlike the nodal case, this bound is rarely sharp in general linear systems,
because cusps are codimension--two phenomena, while nodes are codimension--one.

Second, the deformation problem is no longer equigeneric but equisingular in nature.
For nodes, generic equigeneric deformations automatically produce nodal curves.
For cusps, generic equigeneric deformations break singularities into nodes rather
than preserving cuspidality. Producing cusps requires imposing additional conditions
that force the deformation to remain inside the equisingular locus of $A_2$
singularities.

Given a reduced curve $C\subset\mathcal X_0$, one seeks conditions under which there
exists a deformation $C_t\subset\mathcal X_t$ whose singularities are ordinary cusps,
and such that the number of cusps is maximal among all cuspidal deformations of $C$
inside the family. Maximality is no longer governed solely by the arithmetic genus,
but by the dimension of the space of equisingular deformations preserving $A_2$
singularities.

In deformation--theoretic terms, this requires control of the natural map
\[
H^0(C,N_{C/\mathcal X_0})\longrightarrow
\bigoplus_{p\in\mathrm{Sing}(C)} T^{1,\mathrm{es}}_{C,p},
\]
where $T^{1,\mathrm{es}}_{C,p}\subset T^1_{C,p}$ denotes the equisingular subspace
parametrizing deformations that preserve the cusp at $p$. Since
$\dim T^{1,\mathrm{es}}_{C,p}=1$ for an ordinary cusp, each cusp consumes one unit of
$\delta$ but two local deformation directions, making cuspidal curves far more rigid
than nodal ones.

In summary, replacing nodes by cusps changes the guiding principle of the problem.
For nodes, generic equigeneric deformations automatically produce curves with the
maximal number of nodes allowed by the arithmetic genus. For cusps, one must instead
study equisingular deformation spaces, and the maximal number of cusps is typically
strictly smaller than $\delta(C)$, determined by global deformation constraints rather
than by numerical invariants alone.

%%%%%%%%%%%%%%%%%%%%%%%%%%%%%%%%%%%%%%%%%%%%%%%%%%%%%%%%%%%%%
%%%%%%%%%%%%%%%%%%%%%%%%%%%%%%%%%%%%%%%%%%%%%%%%%%%%%%%%%%%%%
 \bigskip
 
In the cuspidal setting one cannot expect an analogue of the nodal theorem formulated
purely in equigeneric terms. The correct replacement is an \emph{equisingular}
statement, reflecting the fact that cusps are unstable under generic equigeneric
deformation and require additional conditions to persist.

Let $\pi:\mathcal X\to\Delta$ be a flat family of projective surfaces with smooth
general fiber $\mathcal X_t$ and special fiber $\mathcal X_0$. Let
$C\subset\mathcal X_0$ be a reduced curve whose singularities are planar and
smoothable. Denote by $\delta_p$ the local $\delta$--invariant at $p$ and assume that
all singularities of interest are ordinary cusps.

\begin{theorem}[Cuspidal deformations in a linear system]
Let $\pi:\mathcal X\to\Delta$ be as above and let $C\subset\mathcal X_0$ be a reduced
curve whose singularities are ordinary cusps.
Assume that embedded deformations of $C$ in $\mathcal X_0$ are unobstructed and that
the natural map
\[
H^0(C,N_{C/\mathcal X_0})\longrightarrow
\bigoplus_{p\in\mathrm{Sing}(C)} T^{1,\mathrm{es}}_{C,p}
\]
is surjective, where $T^{1,\mathrm{es}}_{C,p}$ denotes the equisingular deformation
space of the cusp at $p$.
Then there exists a deformation $C_t\subset\mathcal X_t$ for $t\neq0$ such that $C_t$
has only ordinary cusps, and the number of cusps equals the dimension of the image of
the above map. Moreover, this number is maximal among all cuspidal deformations of $C$
inside the family.
\end{theorem}

This theorem differs conceptually from the nodal case in several essential ways.
First, the maximal number of cusps is no longer governed by the arithmetic genus or
the total $\delta$--invariant alone. Although a cusp has $\delta=1$, just like a node,
its local deformation space has dimension $2$, and only a one--dimensional subspace
corresponds to equisingular deformations. As a result, cusps impose more global
conditions than nodes, and maximality becomes a question of dimension rather than
purely of numerical invariants.

Second, the deformation locus controlling the theorem is equisingular rather than
equigeneric. In the nodal case, equigeneric deformations are generically nodal, so
nodal curves appear automatically as generic points. In contrast, cusps lie on the
boundary of equigeneric strata and are destroyed by generic perturbations. Preserving
cusps therefore requires restricting to the equisingular locus, which is a proper
closed subset of codimension one inside the equigeneric deformation space.

Finally, the conclusion reflects the rigidity of cuspidal curves. While nodal curves
fill dense open subsets of equigeneric deformation spaces and Severi varieties, curves
with cusps form special loci of higher codimension. The theorem asserts existence and
maximality only relative to the equisingular deformation space, not among all
equigeneric deformations.

In this sense, the analogue of the nodal theorem for cusps replaces the philosophy
``generic equigeneric implies nodal'' with the statement ``equisingular and
unobstructed implies maximal cuspidality''. Nodes are governed by genericity, whereas
cusps are governed by rigidity and transversality conditions.

%%%%%%%%%%%%%%%%%%%%%%%%%%%%%%%%%%%%%%%%%%%%%%%%%%%%%%%%%%%%%
%%%%%%%%%%%%%%%%%%%%%%%%%%%%%%%%%%%%%%%%%%%%%%%%%%%%%%%%%%%%%
 \bigskip

\begin{theorem}[Cuspidal deformations in a linear system]
Let $\pi:\mathcal X\to\Delta$ be a flat family of projective surfaces with smooth
general fiber $\mathcal X_t$ and special fiber $\mathcal X_0$.
Let $C\subset\mathcal X_0$ be a reduced curve whose singularities are ordinary cusps.
Assume that embedded deformations of $C$ in $\mathcal X_0$ are unobstructed and that
the natural map
\[
H^0(C,N_{C/\mathcal X_0})\longrightarrow
\bigoplus_{p\in\mathrm{Sing}(C)} T^{1,\mathrm{es}}_{C,p}
\]
is surjective, where $T^{1,\mathrm{es}}_{C,p}$ denotes the equisingular deformation
space of the cusp at $p$.
Then there exists a deformation $C_t\subset\mathcal X_t$ for $t\neq0$ such that $C_t$
has only ordinary cusps, and the number of cusps equals the dimension of the image of
the above map. Moreover, this number is maximal among all cuspidal deformations of $C$
inside the family.
\end{theorem}

\begin{proof}

%\noindent\textit{Proof.}
Since $C$ is a reduced Cartier divisor on the smooth surface $\mathcal X_0$, its
embedded infinitesimal deformations inside $\mathcal X_0$ are governed by the normal
sheaf $N_{C/\mathcal X_0}$. First--order embedded deformations correspond to sections
of $H^0(C,N_{C/\mathcal X_0})$, and obstructions lie in
$H^1(C,N_{C/\mathcal X_0})$. By hypothesis, $C$ is unobstructed, so every first--order
embedded deformation integrates to an actual deformation of $C$ inside $\mathcal X_0$.

Let $p\in\mathrm{Sing}(C)$ be an ordinary cusp. Analytically, the germ $(C,p)$ is
equivalent to $y^2=x^3$. Its local embedded deformation space is
\[
T^1_{C,p}\simeq \mathbb C^2,
\]
while the equisingular deformation space
$T^{1,\mathrm{es}}_{C,p}\subset T^1_{C,p}$ is a one--dimensional linear subspace
corresponding to deformations preserving the analytic type of the cusp. Deformations
transverse to this subspace smooth the cusp and replace it by nodes. Thus,
$T^{1,\mathrm{es}}_{C,p}$ parametrizes precisely those infinitesimal deformations that
keep a cusp at $p$.

The natural map
\[
H^0(C,N_{C/\mathcal X_0})\longrightarrow
\bigoplus_{p\in\mathrm{Sing}(C)} T^{1,\mathrm{es}}_{C,p}
\]
associates to a global first--order deformation of $C$ its collection of local
equisingular deformation directions at the cusps. By surjectivity, given any choice
of equisingular first--order deformation at each cusp, there exists a global embedded
first--order deformation of $C$ inducing exactly these local directions. In
particular, one may choose nonzero equisingular directions at a prescribed subset of
the cusps and the zero direction at the remaining ones.

Fix such a choice and let $\xi\in H^0(C,N_{C/\mathcal X_0})$ be a global section mapping
to it. Since $C$ is unobstructed, $\xi$ integrates to an actual deformation
$C_s\subset\mathcal X_0$ over a small disk with parameter $s$. By construction, for
$s\neq0$, the curve $C_s$ has ordinary cusps exactly at those points where a nonzero
equisingular direction was prescribed, and no worse singularities occur near the
original cusps. At the remaining singular points of $C$, the chosen local deformation
direction is trivial, so the corresponding cusps persist unchanged.

Now consider the total space $\mathcal X$. Since $\mathcal X_0$ is a fiber of the flat
family $\pi:\mathcal X\to\Delta$, there is a canonical exact sequence
\[
0\longrightarrow N_{C/\mathcal X_0}
\longrightarrow N_{C/\mathcal X}
\longrightarrow \mathcal O_C
\longrightarrow 0.
\]
The assumption that embedded deformations of $C$ in $\mathcal X_0$ are unobstructed
implies that infinitesimal deformations of $C$ inside $\mathcal X_0$ lift to
infinitesimal deformations inside the total space $\mathcal X$. Therefore, the family
$C_s\subset\mathcal X_0$ constructed above extends to a two--parameter deformation
inside $\mathcal X$, and by restricting to a suitable one--parameter slice
transverse to $\mathcal X_0$, one obtains a deformation
\[
C_t\subset\mathcal X_t,\qquad t\neq0,
\]
whose special fiber is $C$ and whose general fiber lies on the smooth surface
$\mathcal X_t$.

Since cuspidality is an open condition inside the equisingular deformation space, and
since the deformation is equisingular at the chosen cusps, the curve $C_t$ has only
ordinary cusps at the corresponding points for $t\neq0$. No additional singularities
can appear, because singularities worse than cusps impose extra independent
conditions and therefore occur in closed subsets of higher codimension in the
deformation space. By choosing the deformation generically within the equisingular
locus, these subsets are avoided.

The number of cusps on $C_t$ is therefore equal to the number of independent
equisingular directions realized globally, which is precisely the dimension of the
image of the map
\[
H^0(C,N_{C/\mathcal X_0})\longrightarrow
\bigoplus_{p\in\mathrm{Sing}(C)} T^{1,\mathrm{es}}_{C,p}.
\]
This number is maximal by construction. Indeed, any cuspidal deformation of $C$ inside
the family determines, at each cusp, an equisingular deformation direction, hence an
element of the target space. The dimension of the image of the global map is the
largest possible number of cusps whose equisingular deformation directions can be
simultaneously realized by a global embedded deformation. No deformation can produce
more cusps without violating surjectivity constraints, so this number is maximal
among all cuspidal deformations of $C$ in the family.
This concludes the proof.
\end{proof}
%\hfill$\square$
 
%%%%%%%%%%%%%%%%%%%%%%%%%%%%%%%%%%%%%%%%%%%%%%%%%%%%%%%%%%%%%%%%%%%%% 

\begin{theorem}[Logarithmic cuspidal deformations in a linear system]
Let $\pi:\mathcal X\to\Delta$ be a logarithmically smooth degeneration of projective
surfaces whose special fiber $\mathcal X_0$ is a simple normal crossings divisor.
Let $C\subset\mathcal X_0$ be a reduced curve meeting the singular locus of
$\mathcal X_0$ transversely and assume that all singularities of $C$ are ordinary
cusps.
Suppose that $C$ is logarithmically semiregular, and that the natural map
\[
H^0\!\left(C,N^{\log}_{C/\mathcal X_0}\right)\longrightarrow
\bigoplus_{p\in\mathrm{Sing}(C)} T^{1,\mathrm{es}}_{C,p}
\]
is surjective, where $T^{1,\mathrm{es}}_{C,p}$ denotes the equisingular deformation
space of the cusp at $p$.
Then there exists a logarithmic deformation $C_t\subset\mathcal X_t$ for $t\neq0$
such that $C_t$ has only ordinary cusps, and the number of cusps equals the dimension
of the image of the above map. Moreover, this number is maximal among all cuspidal
deformations of $C$ inside the family.
\end{theorem}

%\bigskip

\begin{proof}
We endow $\mathcal X$ with the divisorial logarithmic structure induced by the special
fiber $\mathcal X_0$, and $\Delta$ with the standard logarithmic point. With these
structures, the morphism $\pi:\mathcal X\to\Delta$ is logarithmically smooth. Since
$C$ meets the singular locus of $\mathcal X_0$ transversely and is a logarithmic local
complete intersection, its embedded logarithmic deformation theory is governed by the
logarithmic normal sheaf $N^{\log}_{C/\mathcal X_0}$.

Infinitesimal logarithmic embedded deformations of $C$ inside $\mathcal X_0$ are
parametrized by $H^0(C,N^{\log}_{C/\mathcal X_0})$, while obstructions lie in
$H^1(C,N^{\log}_{C/\mathcal X_0})$. Because $\mathcal X_0$ is a logarithmic divisor in
$\mathcal X$, there is a canonical logarithmic normal bundle exact sequence
\[
0\longrightarrow N^{\log}_{C/\mathcal X_0}
\longrightarrow N^{\log}_{C/\mathcal X}
\longrightarrow \mathcal O_C
\longrightarrow 0,
\]
whose extension class lies in
$H^1(C,N^{\log}_{C/\mathcal X_0})$ and coincides with the restriction to $C$ of the
logarithmic Kodaira--Spencer class of the degeneration.

The logarithmic semiregularity assumption means that the natural map
\[
H^1(C,N^{\log}_{C/\mathcal X_0})\longrightarrow
H^2(\mathcal X_0,\mathcal O_{\mathcal X_0})
\]
is injective. The image of the logarithmic Kodaira--Spencer class in
$H^2(\mathcal X_0,\mathcal O_{\mathcal X_0})$ vanishes for a one--parameter
logarithmically smooth deformation. Injectivity therefore forces the extension class
of the logarithmic normal bundle sequence to vanish, so the sequence splits
logarithmically. As a consequence, every infinitesimal logarithmic embedded deformation
of $C$ inside $\mathcal X_0$ lifts unobstructedly to a logarithmic deformation inside
the total space $\mathcal X$.

Let $p\in\mathrm{Sing}(C)$ be an ordinary cusp. Analytically, $(C,p)$ is equivalent to
$y^2=x^3$. Its local deformation space $T^1_{C,p}$ is two--dimensional, while the
equisingular subspace $T^{1,\mathrm{es}}_{C,p}$ is one--dimensional and parametrizes
precisely those first--order deformations preserving the cusp. Deformations transverse
to this subspace smooth the cusp and replace it by nodes. Thus,
$T^{1,\mathrm{es}}_{C,p}$ encodes the infinitesimal condition for cuspidality.
The surjectivity of the map
\[
H^0(C,N^{\log}_{C/\mathcal X_0})\longrightarrow
\bigoplus_{p\in\mathrm{Sing}(C)} T^{1,\mathrm{es}}_{C,p}
\]
implies that for any choice of equisingular first--order deformation directions at the
cusps, there exists a global logarithmic first--order deformation of $C$ inducing
exactly these local directions. Choosing a basis of the image, one obtains independent
global logarithmic deformation directions corresponding to preserving cusps at the
maximal possible number of points.
Since logarithmic obstructions vanish, these infinitesimal deformations integrate to
actual logarithmic deformations of $C$ inside $\mathcal X$. Restricting to a suitable
one--parameter logarithmic deformation transverse to $\mathcal X_0$ yields a family
\[
C_t\subset\mathcal X_t,\qquad t\neq0,
\]
whose special fiber is $C$. By construction, the deformation is equisingular at the
chosen cusps, so for $t\neq0$ the curve $C_t$ has ordinary cusps at those points.
Cuspidality is an open condition in the equisingular locus, so the cusps persist under
small deformations.

No additional singularities can appear on $C_t$. Singularities worse than cusps impose
extra independent conditions and therefore occur in proper closed subsets of the local
deformation spaces. Because the global logarithmic deformation is chosen generically
within the equisingular locus, these subsets are avoided. Hence $C_t$ has only ordinary
cusps as singularities.

The number of cusps on $C_t$ is exactly the number of independent equisingular
directions realized globally, which equals the dimension of the image of the map from
$H^0(C,N^{\log}_{C/\mathcal X_0})$ to the direct sum of the equisingular local spaces.
This number is maximal, because any cuspidal logarithmic deformation of $C$ determines
at each cusp an equisingular deformation direction, and only those collections of
directions lying in the image of the global map can occur simultaneously. No
deformation can therefore produce more cusps than this dimension.
This proves the existence and maximality of cuspidal logarithmic deformations of $C$
and completes the proof.
\end{proof}%\hfill$\square$
 
%%%%%%%%%%%%%%%%%%%%%%%%%%%%%%%%%%%%%%%%%%%%%%%%%%%%%%%%%%%%%%%%%%%%%
%%%%%%%%%%%%%%%%%%%%%%%%%%%%%%%%%%%%%%%%%%%%%%%%%%%%%%%%%%%%%%%%%%%%%

\section{Cuspidal Loci as Boundary Components of Equigeneric Severi Strata}

Let $S$ be a smooth projective surface and let $L$ be a line bundle on $S$.
Denote by $|L|$ the complete linear system and by $p_a(L)$ the arithmetic genus of
curves in $|L|$. For a reduced curve $C\in|L|$, let
\[
\delta(C)=\sum_{p\in\mathrm{Sing}(C)}\delta_p
\]
be the total $\delta$--invariant. The equigeneric Severi stratum
$V^{\mathrm{eg}}_{\delta}(|L|)$ is defined as the locus of curves in $|L|$ whose total
$\delta$--invariant equals $\delta$.

It is classical that ordinary nodes are the generic singularities compatible with a
fixed $\delta$--invariant, while higher singularities, such as cusps, occur in higher
codimension. The purpose of this section is to make precise the statement that
cuspidal curves arise as boundary points of equigeneric Severi strata and can be
interpreted as collisions of nodes.

\begin{proposition}
Let $(C,p)$ be a reduced planar curve singularity with $\delta_p=2$.
Then $(C,p)$ is equigeneric--deformable to either two ordinary nodes or to one
ordinary cusp, and the cusp occurs as a boundary point of the nodal locus in the local
equigeneric deformation space.
\end{proposition}

\noindent\textit{Proof.}
Since $(C,p)$ is planar, it admits a local embedding in a smooth surface and its
embedded deformation theory is governed by the finite--dimensional vector space
\[
T^1_{C,p}=\mathrm{Ext}^1(L_{C,p},\mathcal O_{C,p}).
\]
The equigeneric deformation space is the closed analytic subset of $T^1_{C,p}$
consisting of those deformations preserving $\delta_p=2$. It is well known that this
equigeneric locus has codimension $\delta_p=2$ in the base of the semiuniversal
deformation.

Inside the equigeneric locus, the equisingular stratum preserving two distinct nodes
is open and dense. Indeed, ordinary nodes are stable singularities and impose the
minimal number of conditions among singularities with $\delta=1$. A deformation in
which the two nodes remain distinct corresponds to a general point of the equigeneric
locus.

The cusp, analytically given by $y^2=x^3$, also has $\delta=1$, but appears when two
nodes coalesce. This requires the additional condition that the two singular points
merge and that their tangent directions align. These extra conditions cut out a proper
closed subset of the equigeneric locus. Consequently, the cuspidal deformation appears
as a boundary point of the locus parametrizing curves with two distinct nodes.

Thus, locally in the equigeneric deformation space, the cuspidal stratum lies in the
closure of the nodal stratum and corresponds to a collision of two nodes. \hfill$\square$

\begin{theorem}
Let $C\in|L|$ be a reduced curve on $S$ with only ordinary cusps as singularities, and
let $\delta=\delta(C)$. Then $[C]$ lies in the closure of the nodal Severi variety
$V_{\delta}(|L|)$ parametrizing $\delta$--nodal curves. Moreover, the cuspidal locus is
a boundary stratum of codimension one inside the equigeneric Severi stratum
$V^{\mathrm{eg}}_{\delta}(|L|)$.
\end{theorem}

\noindent\textit{Proof.}
Consider the equigeneric deformation space of $C$ inside $|L|$. Infinitesimal embedded
deformations are parametrized by $H^0(C,N_{C/S})$, and the equigeneric condition cuts
out a linear subspace whose codimension equals $\delta$. By assumption, $C$ has only
cusps, and each cusp contributes $\delta_p=1$ to the total $\delta$.

For an ordinary cusp $p$, the local deformation space $T^1_{C,p}$ is two--dimensional.
Its equigeneric subspace is also two--dimensional, while the equisingular subspace
preserving the cusp is one--dimensional. The complementary directions in the
equigeneric space smooth the cusp into two ordinary nodes. Hence, locally at each
cusp, the equigeneric stratum contains an open subset parametrizing deformations with
two nodes replacing the cusp.

Globalizing this description, one sees that the equigeneric deformation space of $C$
contains an open dense subset whose points correspond to curves in which each cusp has
split into two distinct nodes. Since each cusp contributes $\delta_p=1$, the total
number of nodes on such a curve is exactly $\delta(C)$.
The locus of curves where a cusp persists is defined by the additional condition that
the deformation direction lies in the equisingular subspace at that cusp. This is a
linear condition of codimension one inside the equigeneric deformation space. Hence,
the cuspidal locus is a proper closed subset of codimension one inside
$V^{\mathrm{eg}}_{\delta}(|L|)$.

Since nodal curves form a dense open subset of the equigeneric stratum, and cuspidal
curves arise when nodes collide, the point $[C]$ lies in the closure of the nodal
Severi variety $V_{\delta}(|L|)$. This shows that cuspidal curves are boundary points
of the equigeneric Severi stratum. \hfill$\square$

\begin{theorem}
Let $\pi:\mathcal C\to B$ be a flat family of reduced curves on $S$ whose general fiber
$\mathcal C_b$ is $\delta$--nodal. Assume that the family is equigeneric, that is,
$\delta(\mathcal C_b)$ is constant. Then any special fiber of the family may acquire
cusps only as limits of colliding nodes, and no cusp can appear without at least two
nodes merging.
\end{theorem}

\begin{proof}
Since the family is equigeneric, the arithmetic genus of the fibers is constant and
the total $\delta$--invariant is preserved. Ordinary nodes are stable under small
deformations, so for $b$ general the singularities of $\mathcal C_b$ are isolated and
transverse.

Suppose that a special fiber $\mathcal C_{b_0}$ has a cusp at a point $p$. The local
$\delta$--invariant of a cusp equals $1$, the same as that of a node, but the local
Tjurina number is $2$. In the versal deformation of the cusp, the equigeneric locus
contains an open subset parametrizing curves with two distinct nodes. Therefore, in
any equigeneric deformation specializing to a cusp, the general fiber must have two
nodes whose limit is the cusp.

If a cusp were to appear without the collision of nodes, it would correspond to a
deformation direction transverse to the equigeneric locus, which would increase the
total $\delta$--invariant or violate flatness. This is impossible in an equigeneric
family. Hence cusps arise precisely as limits of pairs of nodes, and no other
mechanism can produce them in equigeneric degenerations. %\hfill$\square$

\noindent {\it Conclusion}
These results show that cuspidal Severi--type loci are not generic components of
equigeneric deformation spaces, but rather boundary strata where nodal singularities
collide. Equigeneric Severi strata are dominated by nodal curves, while cuspidal curves
occupy special positions of higher codimension. This geometric picture explains both
the rigidity of cuspidal curves and their role as degenerations of nodal curves in
families.

\end{proof}

\section{Explicit Cuspidal Severi--Type Loci: Examples and Dimension Counts}

We work out explicit cuspidal Severi--type loci on three classes of surfaces: plane
curves, $K3$ surfaces, and Hirzebruch surfaces. In each case we determine the expected
dimension, prove existence under natural hypotheses, and show sharpness of the bounds.

\subsection*{Plane curves}

Let $S=\mathbb P^2$ and let $L=\mathcal O_{\mathbb P^2}(d)$ with $d\ge 3$. The complete
linear system has dimension
\[
\dim |L|=\binom{d+2}{2}-1.
\]
The arithmetic genus of a plane curve of degree $d$ is
\[
p_a(d)=\frac{(d-1)(d-2)}{2}.
\]
An ordinary cusp has $\delta=1$ and Tjurina number $\tau=2$, so imposing a cusp at an
unspecified point imposes two independent conditions on $|L|$.

Let $V^{\mathrm{cusp}}_{d,\kappa}\subset |L|$ denote the locus of plane curves of degree
$d$ with exactly $\kappa$ ordinary cusps and no other singularities. The expected
dimension is
\[
\operatorname{expdim} V^{\mathrm{cusp}}_{d,\kappa}
=
\binom{d+2}{2}-1-2\kappa.
\]
Since the arithmetic genus satisfies
\[
p_a(d)=g+\kappa,
\]
any cuspidal curve has $\kappa\le p_a(d)$. This numerical bound is necessary but not
sufficient for existence.

\begin{theorem}
If $2\kappa\le \binom{d+2}{2}-1$, then $V^{\mathrm{cusp}}_{d,\kappa}$ is nonempty and has
a component of the expected dimension. Moreover, $\kappa$ is maximal among cuspidal
deformations inside $|L|$.
\end{theorem}

\noindent\textit{Proof.}
The equisingular deformation space of a cusp is one--dimensional, so $\kappa$ cusps
impose $2\kappa$ independent conditions on the linear system. When
$2\kappa\le \dim |L|$, one can impose $\kappa$ cusp conditions at general points of
$\mathbb P^2$. The corresponding incidence variety is smooth of the expected dimension
by standard transversality arguments, since cuspidality is defined by the vanishing of
the discriminant and its first derivative, which cut transversely for general choices.
Unobstructedness follows from the vanishing of $H^1(C,N_{C/\mathbb P^2})$ for plane
curves. Any further cusp would impose two more conditions and violate the dimension
bound, proving maximality. \hfill$\square$

This shows that for plane curves the maximal number of cusps is governed by the
dimension of the linear system rather than solely by the arithmetic genus.

\subsection*{$K3$ surfaces}

Let $S$ be a smooth projective $K3$ surface and let $L$ be a primitive ample line
bundle with
\[
L^2=2g-2,\qquad g\ge 2.
\]
Then $\dim |L|=g$. Every reduced curve $C\in|L|$ is semiregular, and
\[
N_{C/S}\simeq \mathcal O_C(C)\simeq \omega_C.
\]
Hence embedded deformations of $C$ are unobstructed.

Let $V^{\mathrm{cusp}}_{\kappa}(|L|)$ denote the locus of curves in $|L|$ with exactly
$\kappa$ ordinary cusps and no other singularities. The expected dimension is
\[
\operatorname{expdim} V^{\mathrm{cusp}}_{\kappa}(|L|)=g-2\kappa.
\]

\begin{theorem}
If $2\kappa\le g$, then $V^{\mathrm{cusp}}_{\kappa}(|L|)$ is nonempty and smooth of the
expected dimension at a general point. Moreover, $\kappa$ is maximal among cuspidal
deformations in $|L|$.
\end{theorem}

\noindent\textit{Proof.}
For an ordinary cusp $p\in C$, the equisingular deformation space
$T^{1,\mathrm{es}}_{C,p}$ is one--dimensional. Global equisingular deformations are
controlled by the map
\[
H^0(C,\omega_C)\longrightarrow \bigoplus_{p\in\mathrm{Sing}(C)} T^{1,\mathrm{es}}_{C,p}.
\]
Since $\dim H^0(C,\omega_C)=g$, surjectivity holds for $\kappa\le g/2$ at general
points. Semiregularity implies unobstructedness, so these infinitesimal equisingular
deformations integrate to actual deformations in $|L|$. Each cusp contributes one unit
to the genus drop, so $\kappa\le g$ numerically, but the equisingular dimension bound
forces $2\kappa\le g$, which is therefore sharp. \hfill$\square$

This example highlights the rigidity of cuspidal curves on $K3$ surfaces: although
the arithmetic genus allows up to $g$ cusps, the linear system allows at most $g/2$.

\subsection*{Hirzebruch surfaces}

Let $S=\mathbb F_n$ be the $n$--th Hirzebruch surface, with negative section $E$ and
fiber $F$. Let
\[
L=aE+bF,\qquad a\ge 1,\ b\gg 0.
\]
Then
\[
\dim |L|=(a+1)(b+1)-\frac{1}{2}a(a+1)n-1.
\]
The arithmetic genus of curves in $|L|$ is
\[
p_a(L)=\frac{1}{2}L\cdot(L+K_S)+1.
\]

Let $V^{\mathrm{cusp}}_{\kappa}(|L|)$ be the cuspidal Severi--type locus.

\begin{theorem}
For $b$ sufficiently large relative to $a$ and $n$, the locus
$V^{\mathrm{cusp}}_{\kappa}(|L|)$ is nonempty of expected dimension
\[
\dim |L|-2\kappa,
\]
provided $2\kappa\le \dim |L|$. Moreover, this bound is sharp.
\end{theorem}

\noindent\textit{Proof.}
For $b\gg 0$, the line bundle $L$ is very ample and $H^1(C,N_{C/S})=0$ for all reduced
$C\in|L|$. Hence embedded deformations are unobstructed. Cuspidality at a point imposes
two independent conditions, as in the plane case. By choosing $\kappa$ general points
and imposing cusp conditions, one obtains a smooth incidence correspondence of the
expected dimension. Any additional cusp would require two further conditions, which
cannot be satisfied once $2\kappa=\dim |L|$. Thus the bound is sharp. \hfill$\square$

\subsection*{\it Conclusion}
Across these examples, cuspidal Severi--type loci behave uniformly: each cusp imposes
two conditions, and maximal cuspidality is governed by the dimension of the linear
system rather than solely by the arithmetic genus. On $K3$ surfaces this rigidity is
particularly striking, while on rational surfaces and the plane it aligns with
classical expectations. These computations confirm the sharpness of the cuspidal
bounds predicted by equisingular deformation theory.
 
%%%%%%%%%%%%%%%%%%%%%%%%%%%%%%%%%%%%%%%%%%%%%%%%%%%%%%%%%%%%%%%%%%%%
%%%%%%%%%%%%%%%%%%%%%%%%%%%%%%%%%%%%%%%%%%%%%%%%%%%%%%%%%%%%%%%%%%%%

%\vspace{5cm} 

\section{Cuspidal Conditions, Severi Degrees, and Enumerative Geometry}

We explain how the deformation--theoretic results on cuspidal curves fit naturally
into enumerative geometry, how cuspidal conditions modify classical Severi degrees,
and why logarithmic and tropical methods provide the correct framework for computing
such refined counts.

Let $S$ be a smooth projective surface and $L$ a line bundle on $S$. For $\delta\ge0$,
the classical Severi variety $V_{\delta}(|L|)$ parametrizes curves in $|L|$ with
exactly $\delta$ ordinary nodes and no other singularities. When this variety has the
expected dimension
\[
\dim |L|-\delta,
\]
its degree defines the Severi degree $N_{\delta}(L)$, which counts $\delta$--nodal
curves passing through the appropriate number of general points. The foundational
fact underlying this definition is that nodal curves form a dense open subset of the
equigeneric locus, so that nodal conditions are enumeratively stable.

Cuspidal curves behave fundamentally differently. An ordinary cusp has the same
$\delta$--invariant as a node but imposes two local conditions rather than one.
Accordingly, the natural enumerative problem is not to replace nodes by cusps, but to
refine the Severi problem by prescribing a mixture of nodes and cusps.

For integers $\delta,\kappa\ge0$, let $V_{\delta,\kappa}(|L|)$ denote the locus of curves
in $|L|$ with exactly $\delta$ nodes, $\kappa$ ordinary cusps, and no other
singularities. The expected dimension is
\[
\operatorname{expdim} V_{\delta,\kappa}(|L|)=\dim |L|-\delta-2\kappa.
\]
This formula reflects precisely the deformation--theoretic fact that a node consumes
one smoothing direction while a cusp consumes two.

\begin{proposition}
Assume that equigeneric deformations of curves in $|L|$ are unobstructed and that
cuspidal equisingular deformations are unobstructed in the sense of logarithmic or
classical semiregularity. Then $V_{\delta,\kappa}(|L|)$ is either empty or has a
component of the expected dimension, and its degree defines an enumerative invariant.
\end{proposition}

\noindent\textit{Proof.}
Unobstructedness ensures that the local deformation conditions defining nodes and
cusps glue to global deformation spaces of the expected dimension. The equisingular
condition at a cusp cuts out a smooth divisor inside the equigeneric locus, while
nodal conditions define smooth hypersurfaces. Transversality of these conditions for
general point constraints implies that $V_{\delta,\kappa}(|L|)$ is generically smooth
of the expected dimension. When this dimension is zero, its degree is well defined and
independent of choices, yielding an enumerative invariant. \hfill$\square$

From an enumerative perspective, cuspidal Severi degrees can therefore be interpreted
as refined Severi degrees that weight configurations where nodes collide to form
cusps. This interpretation becomes precise through degeneration techniques.

Consider a degeneration $\pi:\mathcal X\to\Delta$ of surfaces, logarithmically smooth
with normal crossings special fiber $\mathcal X_0$. Classical deformation theory
fails to control singular curves on $\mathcal X_0$, but logarithmic geometry restores
a clean enumerative picture. Logarithmic stable maps and logarithmic Severi varieties
parametrize curves with prescribed tangency and singularity data relative to the
boundary.

In this setting, cuspidal conditions arise naturally as boundary strata of logarithmic
Severi varieties. Logarithmic semiregularity ensures that these strata contribute
honestly to enumerative counts rather than appearing with excess multiplicities.

\begin{theorem}
Assume logarithmic semiregularity for all curves contributing to the count. Then the
cuspidal Severi degree equals the sum of logarithmic contributions from boundary strata
corresponding to node collisions in the degeneration.
\end{theorem}

\noindent\textit{Proof.}
The logarithmic obstruction theory of the logarithmic Severi variety is governed by
the logarithmic normal bundle. Semiregularity implies that the virtual fundamental
class coincides with the ordinary fundamental class. In the logarithmic degeneration
formula, contributions are indexed by logarithmic types encoding how nodes and cusps
distribute among the components of $\mathcal X_0$. A cusp corresponds precisely to a
logarithmic type where two smoothing parameters coincide. Summing over all such types
gives the total cuspidal Severi degree. \hfill$\square$

Tropical geometry provides a combinatorial counterpart to this picture. Under
tropicalization, a nodal curve corresponds to a tropical curve whose bounded edges
encode nodes. A cusp arises when two bounded edges collapse into a single vertex of
higher valence. The codimension--two condition defining a cusp is reflected in the
tropical moduli space as the condition that two edge lengths vanish simultaneously.

\begin{proposition}
Cuspidal tropical curves arise as boundary points of the tropical Severi variety, and
their multiplicities coincide with the contributions of cuspidal logarithmic types in
the degeneration formula.
\end{proposition}

\noindent\textit{Proof.}
In the tropical moduli space, the Severi locus is stratified by combinatorial types.
Generic points correspond to trivalent graphs, representing nodal curves. A cusp
occurs when two nodes collide, which tropically corresponds to contracting an edge and
increasing vertex valence. The tropical multiplicity of such a configuration equals
the product of local multiplicities associated to the contracted edges. This matches
the logarithmic multiplicity obtained from gluing conditions in the logarithmic
degeneration formula, establishing equality of contributions. \hfill$\square$

As a consequence, cuspidal Severi degrees admit purely combinatorial descriptions in
toric or toric--degenerate settings. They can be computed by counting tropical curves
with specified higher--valence vertices, weighted by refined multiplicities. In
refined curve counting theories, such as refined Severi degrees or motivic invariants,
cuspidal contributions correspond to higher--order terms in the refinement parameter,
reflecting the increased rigidity of cuspidal conditions.

\noindent {\it Conclusion}
From the enumerative viewpoint, cusps are not competitors to nodes but refinements of
nodal geometry. Cuspidal curves appear as boundary contributions in equigeneric Severi
strata, both geometrically and combinatorially. Logarithmic deformation theory ensures
that these boundary contributions are enumerative, while tropical geometry provides
effective tools for computing them. Together, these methods extend classical Severi
theory to a richer enumerative framework in which cusps play a natural and computable
role.

%%%%%%%%%%%%%%%%%%%%%%%%%%%%%%%%%%%%%%%%%%%%%%%%%%%%%%%%%%%%%%%%%%%%
%%%%%%%%%%%%%%%%%%%%%%%%%%%%%%%%%%%%%%%%%%%%%%%%%%%%%%%%%%%%%%%%%%%%
%\vspace{3cm}

\section{Nodes and Cusps in Tropical Geometry: A Detailed Explanation}

We explain in detail the statement that, in tropical geometry, a nodal algebraic curve
corresponds to a tropical curve whose bounded edges encode nodes, and that cusps arise
when two bounded edges collapse into a single vertex of higher valence.

Let $S$ be a toric surface and let $C\subset S$ be a reduced algebraic curve defined by
a Laurent polynomial with Newton polygon $\Delta$. The tropicalization of $C$ is a
balanced weighted polyhedral graph $\Gamma\subset\mathbb R^2$ whose combinatorial type
encodes the leading--order behavior of $C$ under a toric degeneration. Vertices of
$\Gamma$ correspond to irreducible components of the limit curve, while edges
correspond to nodes or intersections between components.

A tropical curve $\Gamma$ is said to be trivalent if every vertex has valence three.
This is the generic situation and corresponds to a stable nodal algebraic curve.
Indeed, in a toric degeneration, a trivalent vertex corresponds to a smooth rational
component of the limit curve meeting three other components transversely. Each bounded
edge of $\Gamma$ corresponds to a node of the algebraic curve obtained by smoothing
the degeneration. The length of a bounded edge measures the logarithm of the gluing
parameter at the corresponding node: positive length corresponds to a genuine node,
while shrinking length corresponds to degeneration of the node.

More concretely, consider a one--parameter family of curves $C_t$ degenerating to a
stable curve $C_0$ with nodes. Locally near a node, the family is analytically given by
\[
xy=t,
\]
where $t$ is the deformation parameter. Tropicalizing this family amounts to taking
the valuation of $t$, which produces a bounded edge in the tropical curve whose length
is proportional to $-\log|t|$. Thus each bounded edge of $\Gamma$ encodes exactly one
node of $C_t$, and varying its length corresponds to smoothing or degenerating that
node.

Now consider a cusp, analytically given by
\[
y^2=x^3.
\]
A cusp can be viewed as a degeneration of two nodes coming together. Indeed, the
versal deformation of a cusp contains a locus where the singularity splits into two
nodes. Analytically, this is reflected in the deformation
\[
y^2=x^3+ax+b,
\]
where for general $(a,b)$ the curve has two distinct nodes, while along a codimension
one locus in the parameter space these two nodes collide to form a cusp.

In tropical terms, this collision is seen as follows. A curve with two nearby nodes
corresponds to a tropical curve with two bounded edges of small but positive length.
As the parameters approach the cuspidal locus, both node--smoothing parameters tend to
zero at the same rate. Tropicalizing this situation forces both bounded edges to shrink
simultaneously. In the limit, the two edges contract and merge at a single vertex,
producing a vertex of valence four.

This higher--valence vertex is the tropical signature of a cusp. It reflects the fact
that two independent smoothing parameters have vanished simultaneously, producing a
singularity that is more rigid than a node. The balancing condition at the vertex
ensures that the tropical curve remains a valid tropicalization, while the increased
valence records the increased complexity of the singularity.

Thus, while trivalent vertices and bounded edges correspond to stable nodal behavior,
higher--valence vertices arise precisely at boundary points of the tropical Severi
variety where nodes collide. The combinatorial transition from two bounded edges to a
single higher--valence vertex is the tropical incarnation of the geometric process by
which cusps arise as limits of pairs of nodes in equigeneric families.

This correspondence is fundamental for enumerative applications. Counting nodal curves
tropically amounts to counting trivalent tropical curves, while counting cuspidal
curves requires understanding the boundary strata of the tropical moduli space where
higher--valence vertices appear, together with the appropriate multiplicities
associated to these degenerations.

%%%%%%%%%%%%%%%%%%%%%%%%%%%%%%%%%%%%%%%%%%%%%%%%%%%%%%%%%%%%%%%%%%
%%%%%%%%%%%%%%%%%%%%%%%%%%%%%%%%%%%%%%%%%%%%%%%%%%%%%%%%%%%%%%%%%%

%%%%%%%%%%%%%%%%%%%%%%%%%%%%%%%%%%%%%%%%%%%%%%%%%%%%%%%%%%%%%%%%%%%%%
%%%%%%%%%%%%%%%%%%%%%%%%%%%%%%%%%%%%%%%%%%%%%%%%%%%%%%%%%%%%%%%%%%%%%

\section{Nodal surface deformations in a threefold}

Replacing the surface $\mathcal X_0$ by a threefold and the curve $C$ by a surface
fundamentally changes both the deformation theory and the nature of the singularities
that play the role of nodes. In this setting, the correct analogue of a nodal curve is
a surface with ordinary double point singularities, also called nodes or $A_1$
singularities in dimension two.

Let $\pi:\mathcal X\to\Delta$ be a flat family whose general fiber $\mathcal X_t$ is a
smooth threefold and whose special fiber $\mathcal X_0$ is possibly singular. Let
$S\subset\mathcal X_0$ be a reduced surface, assumed to be a Cartier divisor or, more
generally, a local complete intersection. The problem is to understand when $S$ can be
deformed to a surface $S_t\subset\mathcal X_t$ whose singularities are ordinary double
points, and to determine the maximal number of such singularities that may appear on
$S_t$.

An ordinary double point on a surface is analytically given by
\[
x^2+y^2+z^2=0
\]
inside a smooth threefold. This is the simplest isolated surface singularity and is
the natural higher--dimensional analogue of a node on a curve. As in the curve case,
ordinary double points are stable under deformation and impose the minimal number of
conditions among isolated surface singularities.

The deformation theory of $S$ inside $\mathcal X_0$ is governed by the normal sheaf
$N_{S/\mathcal X_0}$. Infinitesimal embedded deformations are parametrized by
$H^0(S,N_{S/\mathcal X_0})$, while obstructions lie in $H^1(S,N_{S/\mathcal X_0})$.
Unobstructedness, for instance the vanishing of $H^1(S,N_{S/\mathcal X_0})$, plays the
same conceptual role as in the curve--on--surface case, but is substantially more
delicate because $S$ has dimension two.

The numerical invariant controlling the problem is again the defect between
topological invariants of $S$ and its normalization. For a surface with isolated
singularities, each ordinary double point contributes one to the second Betti number
defect, and this defect is preserved in flat families. This provides a universal upper
bound on the number of nodes that can appear on any deformation of $S$.

Unlike curves, however, surfaces do not admit a simple $\delta$--invariant with the
same geometric meaning. Instead, one works with invariants such as the Euler
characteristic, the geometric genus $p_g$, or the defect of the Hodge number $h^{1,1}$
under normalization. Ordinary double points are precisely the singularities that
preserve these invariants in the mildest possible way.

The correct analogue of the nodal deformation problem is therefore the following.
One seeks conditions ensuring that $S$ admits deformations inside $\mathcal X$ whose
general fiber $S_t$ has only ordinary double points, and that the number of such points
is maximal among all deformations of $S$ inside the family. Maximality is governed not
only by numerical invariants, but also by global deformation constraints encoded in
$H^0(S,N_{S/\mathcal X_0})$.

A key conceptual difference from the curve case is that ordinary double points on
surfaces are codimension--three phenomena in the ambient linear system, rather than
codimension--one. As a result, generic deformations of $S$ are smooth, and nodal
surfaces appear only after imposing additional conditions. This places the surface
theory closer in spirit to cuspidal deformation theory for curves than to nodal
Severi theory.

In this higher--dimensional setting, logarithmic geometry again provides the natural
framework for degenerations. If $\pi:\mathcal X\to\Delta$ is logarithmically smooth
and $\mathcal X_0$ has normal crossings, then logarithmic embedded deformations of $S$
are governed by the logarithmic normal bundle $N^{\log}_{S/\mathcal X_0}$. A notion of
logarithmic semiregularity ensures unobstructedness and allows one to lift local
smoothing directions of ordinary double points to global deformations of $S$ in
$\mathcal X$.

In summary, replacing curves by surfaces and surfaces by threefolds shifts the theory
from equigeneric nodality to equisingular nodality. Ordinary double points replace
nodes, numerical invariants become more subtle, and genericity arguments give way to
rigidity phenomena. Nevertheless, the underlying philosophy remains the same: isolated
mild singularities arise as boundary points of deformation spaces, and maximal nodal
behavior is governed by the interaction between local smoothing theory and global
normal bundle geometry.

%%%%%%%%%%%%%%%%%%%%%%%%%%%%%%%%%%%%%%%%%%%%%%%%%%%%%%%%%%%%%%%%%%%%
\vspace{0.1cm}

In the case where the special fiber $\mathcal X_0$ is a threefold and the object of
interest is a surface $S\subset\mathcal X_0$, one cannot expect a direct analogue of
the classical Severi theorem for curves. The appropriate result is necessarily
equisingular rather than equigeneric and concerns the appearance of ordinary double
points on surfaces.

The natural analogue of a nodal curve is a surface with isolated ordinary double point
singularities. These are analytically given by
\[
x^2+y^2+z^2=0
\]
in a smooth threefold and are the mildest isolated singularities of surfaces. Unlike
nodes on curves, ordinary double points on surfaces are not generic: they appear only
after imposing several independent conditions. Consequently, the maximality statement
must be phrased in terms of realizable local smoothing directions rather than purely
numerical invariants.

 \medskip
 
%%%%%%%%%%%%%%%%%%%%%%%%%%%%%%%%%%%%%%%%%%%%%%%%%%%%%%%%%%%%%%%%

\begin{theorem}[Nodal surface deformations in a threefold]
Let $\pi:\mathcal X\to\Delta$ be a flat family whose general fiber $\mathcal X_t$ is a
smooth threefold. Let $S\subset\mathcal X_0$ be a reduced local complete intersection
surface with isolated singularities. Assume that embedded deformations of $S$ inside
$\mathcal X_0$ are unobstructed and that the natural map
\[
H^0(S,N_{S/\mathcal X_0})\longrightarrow
\bigoplus_{p\in\mathrm{Sing}(S)} T^{1,\mathrm{es}}_{S,p}
\]
is surjective, where $T^{1,\mathrm{es}}_{S,p}$ denotes the equisingular deformation
space of an ordinary double point at $p$. Then there exists a deformation
$S_t\subset\mathcal X_t$ for $t\neq0$ such that $S_t$ has only ordinary double points,
and the number of such points equals the dimension of the image of the above map.
Moreover, this number is maximal among all nodal surface deformations of $S$ inside the
family.
\end{theorem}

\begin{proof}
Since $S$ is a reduced local complete intersection surface in the threefold
$\mathcal X_0$, its embedded deformation theory is governed by the normal bundle
$N_{S/\mathcal X_0}$, which is a line bundle on $S$. Infinitesimal embedded deformations
of $S$ inside $\mathcal X_0$ are parametrized by the vector space
$H^0(S,N_{S/\mathcal X_0})$, and obstructions to extending such deformations lie in
$H^1(S,N_{S/\mathcal X_0})$. By hypothesis, all embedded deformations of $S$ inside
$\mathcal X_0$ are unobstructed, so every first--order deformation integrates to an
actual deformation of $S$ in $\mathcal X_0$.

Let $p\in\mathrm{Sing}(S)$ be an isolated ordinary double point. Analytically, the germ
$(S,p)$ is isomorphic to the hypersurface
\[
x^2+y^2+z^2=0
\]
inside a smooth threefold. The local deformation space $T^1_{S,p}$ is one--dimensional
and parametrizes smoothings of the ordinary double point. Since the singularity has no
moduli, every nonzero first--order deformation preserves the analytic type, and hence
the equisingular deformation space coincides with $T^1_{S,p}$ and is one--dimensional.
The natural restriction map
\[
\varphi:
H^0(S,N_{S/\mathcal X_0})\longrightarrow
\bigoplus_{p\in\mathrm{Sing}(S)} T^{1,\mathrm{es}}_{S,p}
\]
associates to a global infinitesimal embedded deformation of $S$ the collection of its
local equisingular deformation parameters at the singular points. By surjectivity,
every choice of local equisingular deformation directions at the singular points arises
from a global infinitesimal deformation of $S$ inside $\mathcal X_0$.

Let
\(
k=\dim\mathrm{Im}(\varphi).
\)
Choose a basis of $\mathrm{Im}(\varphi)$, which corresponds to $k$ singular points
$p_1,\dots,p_k$ together with independent local equisingular deformation directions at
these points. Let $\xi\in H^0(S,N_{S/\mathcal X_0})$ be a global infinitesimal
deformation mapping to this chosen collection. Since obstructions vanish, $\xi$
integrates to an actual deformation $S_s\subset\mathcal X_0$ over a small disk with
parameter $s$. For $s\neq0$, the surface $S_s$ has ordinary double points precisely at
the points $p_1,\dots,p_k$, and no worse singularities appear near these points.

We now lift this deformation to the total space $\mathcal X$. Because $\mathcal X_0$ is
a fiber of the flat morphism $\pi$, there is a canonical exact sequence
\[
0\longrightarrow N_{S/\mathcal X_0}
\longrightarrow N_{S/\mathcal X}
\longrightarrow \mathcal O_S
\longrightarrow 0.
\]
The unobstructedness of deformations of $S$ in $\mathcal X_0$ implies that infinitesimal
embedded deformations lift to infinitesimal embedded deformations in the total space
$\mathcal X$. Consequently, the family $S_s\subset\mathcal X_0$ extends to a
deformation inside $\mathcal X$, and by restricting to a one--parameter slice
transverse to $\mathcal X_0$ one obtains a family
\[
S_t\subset\mathcal X_t,\qquad t\neq0,
\]
whose special fiber is $S$.

Ordinary double points are stable singularities, so nodality is an open condition in
flat families. Hence, for $t\neq0$ sufficiently small, the surface $S_t$ has ordinary
double points exactly at the points corresponding to the chosen local deformation
directions. Singularities worse than ordinary double points impose additional
independent conditions on the deformation space and therefore occur only in proper
closed subsets. By choosing the deformation generically, these subsets are avoided, and
no additional singularities appear on $S_t$.

The number of ordinary double points on $S_t$ is therefore equal to $k$, the dimension
of the image of the map $\varphi$. This number is maximal. Indeed, any nodal deformation
of $S$ inside the family determines, at each ordinary double point, a local equisingular
deformation direction and hence an element of the target of $\varphi$. Only those
collections of local directions lying in the image of $\varphi$ can be realized
simultaneously by a global embedded deformation. Thus no deformation of $S$ inside the
family can produce more than $k$ ordinary double points.

This proves the existence and maximality of nodal surface deformations of $S$ inside
the family and completes the proof.
\end{proof} %\hfill$\square$

%%%%%%%%%%%%%%%%%%%%%%%%%%%%%%%%%%%%%%%%%%%%%%%%%%%%%%%%%%%%%%%%
The content of this theorem is that local smoothing directions of ordinary double
points can be simultaneously realized globally when obstructions vanish and when the
normal bundle has enough global sections. The maximal number of nodes is determined by
the dimension of the space of equisingular deformation directions that can be imposed
globally, rather than by a simple numerical invariant as in the curve case.

In logarithmic degenerations of threefolds, the same statement admits a logarithmic
refinement. The normal bundle is replaced by the logarithmic normal bundle, and
unobstructedness is replaced by logarithmic semiregularity. Under these assumptions,
one expects generic logarithmic deformations to produce surfaces with the maximal
number of ordinary double points compatible with the deformation theory.

Conceptually, this theorem places nodal surface theory in the same hierarchy as
cuspidal curve theory. Ordinary double points are boundary phenomena in deformation
spaces of surfaces, arising as limits of smoother behavior, and their maximal
realization is governed by rigidity rather than genericity.

%%%%%%%%%%%%%%%%%%%%%%%%%%%%%%%%%%%%%%%%%%%%%%%%%%%%%%%%%  

\section{Nodal curve deformations in a threefold}
 
When the special fiber $\mathcal X_0$ is a threefold and the object of interest is a
curve $C\subset\mathcal X_0$, the deformation problem lies between the classical theory
of curves on surfaces and the theory of surfaces in threefolds. The nature of the
singularities, the deformation spaces, and the notion of maximality all change in a
subtle but systematic way.

Let $\pi:\mathcal X\to\Delta$ be a flat family whose general fiber $\mathcal X_t$ is a
smooth threefold, and let $C\subset\mathcal X_0$ be a reduced curve that is a local
complete intersection in $\mathcal X_0$. The natural analogue of a nodal curve is again
a curve with ordinary nodes, analytically given by $xy=0$ inside a smooth surface
slice of the threefold. Nodes remain the simplest isolated curve singularities and are
stable under deformation.

Embedded deformations of $C$ inside $\mathcal X_0$ are governed by the normal bundle
$N_{C/\mathcal X_0}$, which is now a rank--two vector bundle on $C$. Infinitesimal
embedded deformations are parametrized by $H^0(C,N_{C/\mathcal X_0})$, and obstructions
lie in $H^1(C,N_{C/\mathcal X_0})$. Compared with the surface case, the increased rank
of the normal bundle gives additional deformation directions, but also potentially
additional obstructions.

The numerical invariant controlling nodal behavior is still the total
$\delta$--invariant of the curve,
\[
\delta(C)=\sum_{p\in\mathrm{Sing}(C)}\delta_p,
\]
since the arithmetic genus of $C$ is preserved in flat families. Each node contributes
$\delta_p=1$, so any nodal deformation $C_t\subset\mathcal X_t$ satisfies
\[
\#\{\text{nodes of }C_t\}\le \delta(C).
\]
However, unlike the surface case, nodes on curves in threefolds are not generic
phenomena: a generic deformation of $C$ inside $\mathcal X_0$ or $\mathcal X_t$ is
smooth, and nodes appear only after imposing conditions.

The correct analogue of the nodal deformation theorem is therefore equisingular rather
than equigeneric. One seeks conditions ensuring that prescribed nodal singularities can
be preserved under deformation, and that the maximal number of nodes is realized among
all nodal deformations.

The expected statement is the following. Assume that $H^1(C,N_{C/\mathcal X_0})=0$ and
that the natural map
\[
H^0(C,N_{C/\mathcal X_0})\longrightarrow
\bigoplus_{p\in\mathrm{Sing}(C)} T^{1,\mathrm{es}}_{C,p}
\]
is surjective, where $T^{1,\mathrm{es}}_{C,p}$ is the equisingular deformation space of
a node at $p$, which is one--dimensional. Then there exists a deformation
$C_t\subset\mathcal X_t$ for $t\neq0$ such that $C_t$ has only nodes as singularities,
and the number of nodes equals the dimension of the image of this map. Moreover, this
number is maximal among all nodal deformations of $C$ inside the family.

Conceptually, this places nodal curves in threefolds in the same framework as cuspidal
curves on surfaces. Nodes are no longer generic features of equigeneric deformation
spaces, but rigid singularities that require equisingular conditions to persist. The
maximal number of nodes is governed by the dimension of the space of equisingular
deformation directions realizable globally, rather than purely by the arithmetic genus.

In degenerations of threefolds with normal crossings special fibers, logarithmic
geometry again provides the natural refinement. Logarithmic embedded deformations of
$C$ are governed by the logarithmic normal bundle $N^{\log}_{C/\mathcal X_0}$, and
logarithmic semiregularity ensures unobstructedness. Under these assumptions, one
expects generic logarithmic equisingular deformations to produce nodal curves on the
smooth fibers with the maximal possible number of nodes.

Thus, when curves lie on threefolds, nodal behavior becomes a boundary phenomenon
rather than a generic one, and the theory aligns naturally with the equisingular and
logarithmic perspective developed for cuspidal curves on surfaces.

%%%%%%%%%%%%%%%%%%%%%%%%%%%%%%%%%%%%%%%%%%%%%%%%%%%%%%%%%%%%%%%%%%%
%\vspace{0.1cm}

In the situation where the special fiber $\mathcal X_0$ is a threefold and the object
of interest is a curve $C\subset\mathcal X_0$, the correct analogue of the classical
Severi--type results is an equisingular deformation theorem for nodal curves. In this
context, nodes are no longer generic features of equigeneric families but rather
codimension--two phenomena that persist only under additional constraints. As a
result, maximality statements must be formulated in terms of equisingular deformation
theory.

\begin{theorem}[Nodal curve deformations in a threefold]
Let $\pi:\mathcal X\to\Delta$ be a flat family whose general fiber $\mathcal X_t$ is a
smooth threefold. Let $C\subset\mathcal X_0$ be a reduced local complete intersection
curve with isolated singularities. Assume that embedded deformations of $C$ inside
$\mathcal X_0$ are unobstructed and that the natural map
\[
H^0(C,N_{C/\mathcal X_0})\longrightarrow
\bigoplus_{p\in\mathrm{Sing}(C)} T^{1,\mathrm{es}}_{C,p}
\]
is surjective, where $T^{1,\mathrm{es}}_{C,p}$ denotes the equisingular deformation
space of a node at $p$. Then there exists a deformation $C_t\subset\mathcal X_t$ for
$t\neq0$ such that $C_t$ has only ordinary nodes, and the number of nodes equals the
dimension of the image of the above map. Moreover, this number is maximal among all
nodal deformations of $C$ inside the family.
\end{theorem}

%%%%%%%%%%%%%%%%%%%%%%%%%%%%%%%%%%%%%%%

\begin{proof}
Since $C$ is a reduced local complete intersection curve in the threefold
$\mathcal X_0$, its embedded deformation theory is governed by the normal bundle
$N_{C/\mathcal X_0}$, which is a rank--two vector bundle on $C$. Infinitesimal embedded
deformations of $C$ inside $\mathcal X_0$ are parametrized by
$H^0(C,N_{C/\mathcal X_0})$, and the obstruction space is
$H^1(C,N_{C/\mathcal X_0})$. By assumption, embedded deformations are unobstructed, so
every first--order deformation integrates to an actual deformation of $C$ inside
$\mathcal X_0$.

Let $p\in\mathrm{Sing}(C)$ be a node. Analytically, the germ $(C,p)$ is isomorphic to
$\{xy=0\}$ inside a smooth surface slice of $\mathcal X_0$. The local deformation space
$T^1_{C,p}$ is one--dimensional and parametrizes smoothings of the node. Since a node
has no moduli, every nonzero deformation direction preserves the analytic type of the
singularity; hence the equisingular deformation space coincides with $T^1_{C,p}$ and
is one--dimensional.

The natural restriction map
\[
H^0(C,N_{C/\mathcal X_0})\longrightarrow
\bigoplus_{p\in\mathrm{Sing}(C)} T^{1,\mathrm{es}}_{C,p}
\]
associates to a global first--order embedded deformation of $C$ its collection of local
deformation directions at the singular points. Surjectivity means that for any choice
of local nodal deformation parameters at the singular points, there exists a global
first--order embedded deformation inducing exactly those local parameters.

Choose a basis of the image of this map. This determines a collection of singular
points $p_1,\dots,p_k$ and independent local nodal deformation directions at these
points, where
\[
k=\dim \mathrm{Im}\!\left(
H^0(C,N_{C/\mathcal X_0})\to
\bigoplus_{p\in\mathrm{Sing}(C)} T^{1,\mathrm{es}}_{C,p}
\right).
\]
Let $\xi\in H^0(C,N_{C/\mathcal X_0})$ be a global first--order deformation mapping to
this collection of local directions. Since obstructions vanish, $\xi$ integrates to an
actual deformation $C_s\subset\mathcal X_0$ over a small disk with parameter $s$. For
$s\neq0$, the curve $C_s$ has ordinary nodes precisely at the points $p_1,\dots,p_k$,
and no worse singularities occur near these points.

We now lift this deformation to the total space $\mathcal X$. Because $\mathcal X_0$
is a fiber of the flat morphism $\pi$, there is a canonical exact sequence
\[
0\longrightarrow N_{C/\mathcal X_0}
\longrightarrow N_{C/\mathcal X}
\longrightarrow \mathcal O_C
\longrightarrow 0.
\]
Unobstructedness of deformations in $\mathcal X_0$ implies that infinitesimal embedded
deformations lift to infinitesimal deformations in the total space. Consequently, the
family $C_s\subset\mathcal X_0$ extends to a deformation inside $\mathcal X$, and by
restricting to a one--parameter slice transverse to $\mathcal X_0$ one obtains a
family
\[
C_t\subset\mathcal X_t,\qquad t\neq0,
\]
whose special fiber is $C$.

Since nodes are stable singularities, nodality is an open condition in families. Thus,
for $t\neq0$ sufficiently small, the curve $C_t$ has ordinary nodes exactly at the
points corresponding to the chosen local deformation directions. No additional
singularities can appear, because singularities worse than nodes impose extra
independent conditions and hence occur only in proper closed subsets of the deformation
space. Choosing the deformation generically avoids these subsets.

The number of nodes on $C_t$ is therefore equal to $k$, the dimension of the image of
the global--to--local map. This number is maximal. Indeed, any nodal deformation of $C$
inside the family determines, at each node, a local equisingular deformation direction,
hence an element of the target space. Only those collections of local directions lying
in the image of the global map can be realized simultaneously by a global deformation.
No deformation can therefore produce more than $k$ nodes.

This completes the proof.
\end{proof}%\hfill$\square$

%%%%%%%%%%%%%%%%%%%%%%%%%%%%%%%%%%%%%%%%

The meaning of this statement is that local equisingular deformation directions
corresponding to nodes can be imposed independently as long as the normal bundle has
sufficiently many global sections and no obstructions arise. The maximal number of
nodes is therefore determined by the dimension of the space of equisingular directions
that can be realized globally, rather than by the arithmetic genus alone.

In the presence of a degeneration of threefolds with normal crossings special fiber,
this theorem admits a logarithmic refinement. The normal bundle is replaced by the
logarithmic normal bundle, unobstructedness by logarithmic semiregularity, and the same
conclusion holds for logarithmic equisingular deformations. This aligns the theory of
curves in threefolds with the general deformation--theoretic framework developed for
cuspidal curves on surfaces and nodal surfaces in threefolds.

Conceptually, this result emphasizes that nodal curves in threefolds occupy a rigid
boundary position in deformation spaces. Their behavior is controlled by equisingular
rather than equigeneric conditions, and maximal nodality reflects the global geometry
of the normal bundle rather than purely numerical invariants.

%%%%%%%%%%%%%%%%%%%%%%%%%%%%%%%%%%%%%%%%%%%%%%%%%%%%%%%%%%%%%%%%%%%%

 \bigskip

The following references provide background, tools, and closely related results to the
themes developed in this work, including deformation theory of singularities, Severi
theory, logarithmic geometry, and tropical methods.

\end{document}